\documentclass[leqno,10pt]{article}

\usepackage{amssymb, amsmath, amsthm, ascmac, mathrsfs}
\usepackage[all]{xy}

\usepackage{color}
\makeatletter \newcommand{\dashedrightarrow}[1][2pt]{%
\settowidth{\@tempdima}{$\rightarrow$}\rightarrow
\makebox[-\@tempdima]{\hskip-1.5ex\color{white}\rule[0.5ex]{#1}{1pt}}
\phantom{\rightarrow}
} 
\makeatother 

\theoremstyle{plain}
\newtheorem{thm}{Theorem}[section]
\newtheorem*{thm*}{Theorem}

\newtheorem{lmm}[thm]{Lemma}
\newtheorem{crl}[thm]{Corollary}

\theoremstyle{remark}

\theoremstyle{definition}

\renewcommand{\L}{\Lambda}

\newcommand{\ab}{\mathrm{ab}}

\newcommand{\Char}{\operatorname{char}}

\newcommand{\Z}{\mathbb{Z}}

\newcommand{\inj}{\hookrightarrow}
\newcommand{\surj}{\twoheadrightarrow}

\newcommand{\Coker}{\operatorname{Coker}}
\newcommand{\Ker}{\operatorname{Ker}}
\newcommand{\isomto}{\stackrel{\simeq}{\longrightarrow}}

\newcommand{\ol}[1]{\overline{#1}}

\newcommand{\Cf}{\textit{cf.}\;}

\def\sn{\smallskip\noindent}

\def\ssm{\smallsetminus}

\newcommand{\Gm}{\mathbb{G}_{m}}
\newcommand{\Ga}{\mathbb{G}_{a}}

\newcommand{\CH}{\operatorname{CH}}

\newcommand{\Spec}{\operatorname{Spec}}

\renewcommand{\div}{\operatorname{div}}
\newcommand{\red}{\mathrm{red}}

\newcommand{\Cbar}{\overline{C}}
\renewcommand{\mod}{\operatorname{mod}}

\renewcommand{\L}{\mathscr{L}}
\newcommand{\M}{\mathscr{M}}
\newcommand{\N}{\mathscr{N}}

\newcommand{\C}{\mathscr{C}}
\newcommand{\K}{\mathscr{K}}
\newcommand{\sCH}{{\C}\!H}

\newcommand{\Gal}{\operatorname{Gal}}
\renewcommand{\O}{\mathscr{O}}

\newcommand{\otimesM}{\!\stackrel{M}{\otimes}\!}

\newcommand{\Tr}{\operatorname{Tr}}

\title{On certain products of algebraic groups over a finite field}

\author{Toshiro Hiranouchi}

\begin{document}
\maketitle
\begin{abstract}      
Let $G_1,\ldots ,G_n$ be smooth connected and commutative algebraic groups 
over a finite field $F$. 
We show that 
$T(G_1,\ldots, G_n)(\Spec F) = 0$ 
if $n\ge 2$, 
where $T(G_1,\ldots, G_n)$ is 
the reciprocity functor 
of Ivorra-R\"ulling 
associated with those algebraic groups as reciprocity functors.  
We also show that 
the finiteness of 
$(G_1\otimesM \cdots \otimesM G_n)(\Spec F)$
the tensor product 
of $G_1,\ldots ,G_n$ in the category of Mackey functors. 
We apply this to prove that, 
for a product of open curves, 
the finiteness of 
the relative Chow group and 
an abelian fundamental group 
which classifies abelian coverings with 
bounded ramification along the boundary. 
\end{abstract}


\section{Introduction}
A Mackey functor over a perfect field $F$ 
in the sense of \cite{Kahn92b} is a co- and contravariant functor 
from the category of \'etale schemes over $F$ to the category of abelian groups. 
A smooth connected and commutative algebraic group $G$ over the field $F$ 
is regarded as a Mackey functor by the correspondence 
$x \mapsto G(x)$. 
Such algebraic group $G$ can be extended to 
a Nisnevich sheaf with transfers on the category of regular schemes over $F$ with dimension $\le 1$. 
Furthermore, it satisfies the following condition which is the so-called reciprocity law 
(\cite{IR}, Prop.\ 2.2.2): 
For any open (=non-proper) regular connected curve $C$ over $F$ and a section $a\in G(C)$ 
there exists an effective divisor $D$ on the smooth compactification $\Cbar$ of $C$ 
with support in the boundary $\Cbar \ssm C$  such that  
$$
 \sum_{x\in C} v_x(f) \Tr_{x/x_C}s_x(a) =  0
$$
for any $f \neq 0$ in the function field $F(C)$ of $C$ such that $f\equiv 1\, \mod D$, 
that is, $\div(f-1) \ge D$ as Weil divisors, 
where $v_x$ is the valuation at $x$, 
$s_x:G(C) \to G(x)$ is the pull-back along the natural inclusion $x\inj C$ and 
$\Tr_{x/x_C} :G(x) \to G(x_C)$ is the push-forward 
along the finite map $x\to x_C := \Spec H^0(\Cbar, \O_{\Cbar})$. 
F.~Ivorra and K.~R\"ulling \cite{IR} has 
introduced the notion of a {\it reciprocity functor} 
as a Nisnevich sheaf with transfers on the category of regular schemes over $F$ with dimension $\le 1$ 
satisfying several axioms including the one like the reciprocity law as above. 
They have also introduced a ``product''  
$T(\M_1,\ldots ,\M_n)$ 
associated to reciprocity functors $\M_1,\ldots ,\M_n$ 
in the quasi-abelian category of reciprocity functors 
(for the precise definition of the ``product'', see \cite{IR} Def.\ 4.2.3). 
By the very construction of the product $T$, 
as a Mackey functor $T(\M_1,\ldots , \M_n)$ 
is a quotient of the tensor product $\M_1 \otimesM \cdots \otimesM \M_n$ 
as Mackey functors (for the definition, see (\ref{eq:M}) in the next section). 
Hence we have a canonical surjection 
$$
  (\M_1 \otimesM \cdots \otimesM \M_n)(\Spec F) \surj T(\M_1\ldots , \M_n)(\Spec F). 
$$
Although 
the tensor product $\,\otimesM\,$ 
gives a structure of a symmetric monoidal category 
in the abelian category of Mackey functors, 
it is not known that whether this product $T$  
satisfies the associativity and 
then gives a monoidal structure or not. 
However, this product coincides with the $K$-group 
of B.~Kahn and T.~Yamazaki 
\cite{KY12} 
if we take homotopy invariant Nisnevich sheaves with transfers 
as reciprocity functors. 
In particular, we obtain 
an isomorphism 
$$
T(G_1,\ldots , G_n)(\Spec F) \simeq K(F; G_1,\ldots ,G_n),
$$ 
for semi-abelian varieties $G_1,\ldots ,G_n$ over $F$, 
where $K(F;G_1,\ldots ,G_n)$ is Somekawa's $K$-group \cite{Som90} 
which was limited on considering only semi-abelian varieties.  
For semi-abelian varieties $G_1,\ldots ,G_n$ over a {\it finite field} $F$,
B.~Kahn in \cite{Kahn92} showed 
that 
\begin{equation}
   K(F;G_1,\ldots ,G_n) = (G_1\otimesM \cdots \otimesM G_n)(\Spec F) =0 
\label{eq:Kahn}
\end{equation}
if $n>1$. 
Because of the isomorphism (\cite{Som90}, Thm.\ 1.4) 
$$
  K(F; \overbrace{\Gm,\ldots ,\Gm}^n) \isomto K_n^M(F),
$$
where $K_n^M(F)$ is the Milnor $K$-group of the field $F$, 
these results generalize the classical fact 
that $K_n^M(F) =0$ if $F$ is a finite field $F$ and $n>1$. 
For algebraic groups $G_1,G_2$ which may contain unipotent part, 
it is easy to see that 
$(G_1 \otimesM G_2)(\Spec F)$ may not be trivial.  
In this note, we shall show the following theorem. 

\begin{thm}[Thm.\ \ref{thm:S}, \ref{thm:M}]
Let $G_1,\ldots, G_n$ be smooth commutative and connected algebraic groups over 
a finite field $F$ with characteristic $\neq 2$ for $n>1$. 
Then we have
$$
  T(G_1,\ldots ,G_n)(\Spec F) = 0,\quad (G_1\otimesM \cdots \otimesM G_n)(\Spec F) \mbox{ is finite}.
$$
\end{thm}

As an application of (\ref{eq:Kahn}), 
the class field theory of 
a product of projective smooth curves over a finite field, 
a special case of 
the higher dimensional class field theory of S.~Bloch, K.~Kato and S.~Saito (e.g., \cite{KS83b}) 
is reduced from the 
classical (unramified) class field theory  
(= class field theory of curves over a finite field) and 
Lang's theorem; the reciprocity map 
on a normal variety over a finite field 
has dense image. 
In Section \ref{sec:app}, we will pursue 
related results on the (ramified)  
class field theory of a product of open (=non-proper) curves as a byproduct 
of the above theorem. 
In particular, we obtain a finiteness of 
the relative Chow group $\CH_0(X,D)$ 
for a product of smooth curves $X = X_1\times \cdots \times X_n$ 
over a finite field 
and an effective divisor $D$ 
on the smooth compactification $\ol{X}$ of $X$ with support in $\ol{X} \ssm X$ 
(Thm.\ \ref{thm:CH}).

Throughout this note, 
we mean 
by an {\it algebraic group} 
a smooth connected and commutative group scheme over a field. 
For a field $F$, we denote by $\Char(F)$ the characteristic of $F$. 

\medskip\noindent
{\it Acknowledgments.} 
A part of this note was written 
during a stay of the author at the Duisburg-Essen university. 
He thanks the institute for its hospitality. 
Most of what the author know about  
relative Chow groups and Albanese varieties 
from Henrik Russell.

\section{Finiteness}
First we recall the definition and some properties of 
the product $\,\otimesM\,$ 
in the category of Mackey functors 
and those of the product of reciprocity functors 
following \cite{IR}.

Let $F$ be a perfect field. 
We call a morphism $x\to \Spec F$ 
a {\it finite point}\/ if $x = \Spec E$ for some finite field extension of $F$. 
A Mackey functor $M$ over $F$ 
in the sense of \cite{Kahn92b} is determined by 
its value $M(x)$ on finite points $x\to \Spec F$. 
For Mackey functors $M_1,\ldots ,M_n$, 
the product $M_1 \otimesM \cdots \otimesM M_n$ 
called the {\it Mackey product}\/ 
is defined as follows.  
For any finite point $x\to \Spec F$, 
\begin{equation}
\label{eq:M}
  (M_1 \otimesM \cdots \otimesM M_n)(x) := \left(\bigoplus_{y\to x:\ \mbox{finite}} M_1(y)\otimes_{\Z}\cdots \otimes_{\Z} M_n(y)\right)\bigg/R(x)
\end{equation}
where $y\to x$ runs all finite points over $x$, 
and $R(x)$ is the subgroup generated by elements of the following form:  
For any morphism $j: y' \to y$ of finite points over $x$, 
and if $a_{i_0}' \in M_{i_0}(y')$ and $a_i \in M_i(y)$ 
for $i\neq i_0$, then 
\begin{equation}
  j^{\ast}(a_1) \otimes \cdots \otimes a_{i_0}' \otimes \cdots \otimes j^{\ast}(a_n) 
    - a_1\otimes \cdots \otimes j_{\ast}(a_{i_0}')\otimes \cdots \otimes a_n \in R(x), \label{eq:R}
\end{equation}
where $j^{\ast}$ and $j_{\ast}$ are 
the pull-back and the push-forward along $j$ respectively. 
We write $\{a_1,\ldots ,a_n\}_{y/x}$ 
for the image of $a_1\otimes \cdots \otimes a_n \in M_1(y)\otimes_{\Z} \cdots \otimes_{\Z} M_n(y)$ 
in the product $(M_1\otimesM \cdots \otimesM M_n)(x)$. 
Using this symbol, the above relation (\ref{eq:R}) defining $R(x)$ above 
gives the following equation which is often called the {\it projection formula}: 
\begin{equation}
  \label{eq:P}
  \{j^{\ast}(a_1) ,\ldots ,a_{i_0}' ,\ldots, j^{\ast}(a_n)\}_{y'/x} 
    = \{a_1 ,\ldots, j_{\ast}(a_{i_0}'),\ldots,  a_n\}_{y/x}.
\end{equation}

The Mackey product (\ref{eq:M}) satisfies the following properties: 

\begin{itemize}
\item[(M1)] The Mackey product $\,\otimesM\,$ 
gives a tensor product 
in the abelian category 
of the Mackey functors. 
Its unit is the constant Mackey functor $\Z$. 
In particular, the product commutes with the direct sum $\oplus$ and 
satisfies the associativity: $M_1\otimesM M_2 \otimesM M_3 \simeq (M_1\otimesM M_2) \otimesM M_3$.

\item[(M2)] 
The product $-\otimesM M$ is right exact for any Mackey functor $M$. 

\item[(M3)] 
For any finite point $j:x' \to x$, 
the push-forward 
$j_{\ast} : (M_1\otimesM \cdots \otimesM M_n)(x') \to  
(M_1\otimesM \cdots \otimesM M_n)(x)$  
along $j$  
is given by $j_{\ast}(\{a_1,\ldots ,a_n\}_{y'/x'}) = \{a_1,\ldots, a_n\}_{y'/x}$ on symbols. 
\end{itemize}

\begin{lmm}
\label{lem:UA} 
Let $G$ be a unipotent (smooth and commutative) algebraic group 
over $F$
and $A$ a semi-abelian variety over $F$. 
If $F$ is a perfect field of $\Char(F) = p>0$, 
we have $G \otimesM A = 0$.
\end{lmm}
\begin{proof}
The unipotent group $G$ 
has a composition series: 
$$
  0 = G^r \subset \cdots \subset G^1 \subset G,
$$
each $G^i/G^{i+1}$ being isomorphic to $\Ga$. 
By the right exactness (M2), 
it is enough to show $(\Ga \otimesM A) = 0$.
By (M3) above, the assertion is reduced to showing $\{a, b\}_{x/x} = 0$ 
for any $a\in \Ga(x), b\in A(x)$. 
There exists a finite point $j:x'\to x$ such that 
$j^{\ast}(b) = pb'$ for some $b'\in A(x')$. 
Since the trace map (= the push-forward map on $\Ga$) 
$j_{\ast} = \Tr_{x'/x}:\Ga(x') \to \Ga(x)$ 
is surjective, we obtain 
\begin{align*}
  \{a,b\}_{x/x} 
&= \{\Tr_{x'/x}(a'), b\}_{x/x}\quad \mbox{for some $a'\in \Ga(x')$}\\ 
&= \{a', j^{\ast}b\}_{x'/x}\quad \mbox{by the projection formula (\ref{eq:P})}\\
&= \{a', pb'\}_{x'/x} \\
&= 0.
\end{align*}
The assertion follows from this.
\end{proof}

As noted before, 
a reciprocity functor is a Nisnevich sheaf with transfers 
on the category of regular schemes with dimension $\le 1$ 
satisfying several axioms (\cite{IR}, Def.\ 1.5.1). 
As examples, 
a constant Nisnevich sheaf, algebraic groups, 
the Milnor $K$-theory $\K_n^M$, 
the higher Chow group $\sCH_0(X,n)$ for some scheme $X$, 
Suslin's singular homology group $h_0(X)$  
and the sheaf of the absolute K\"ahler differentials $\Omega^n$ 
given by $X\mapsto \Omega_{X/\Z}^n$ 
are reciprocity functors (\cite{IR}, Sect.\ 2). 
Note that any reciprocity functor gives a Mackey functor 
by restricting to the category of points over $F$. 
It is known that 
the category of reciprocity functors forms a quasi-abelian category, 
especially, an exact category. 
Hence the notion of an {\it admissible exact sequence}  
$0\to \M' \to \M \to \M'' \to 0$ 
and a right exact functor are defined in the category of reciprocity functors. 
The ``product'' $T(\M_1,\ldots ,\M_n)$ for reciprocity functors $\M_1,\ldots , \M_n$ 
is a reciprocity functor and satisfies some functorial properties. 
\begin{itemize}
\item[(R1)] There are functorial isomorphisms 
$$
  T(\M_1,\ldots, \M_i,\ldots, \M_j ,\ldots, \M_n)\simeq T(\M_1,\ldots, \M_j,\ldots ,\M_i,\ldots, \M_n)
$$
and 
$$
T(\M_1,\ldots ,\M_i\oplus \M_i', \ldots ,\M_n) 
\simeq T(\M_1,\ldots ,\M_i, \ldots ,\M_n) \oplus 
T(\M_1,\ldots ,\M_i', \ldots ,\M_n).
$$ 
There is an admissible epimorphism 
$$
  T(\M_1,\M_2,\M_3)\surj T(T(\M_1,\M_2),\M_3).
$$
However, it is not known whether this map becomes an isomorphism. 

\item[(R2)] 
The functor $T(\M_1,\ldots, \M_{n-1},-)$ is right exact  
(\cite{IR}, Cor.\ 4.2.9). 

\item[(R3)] Assume $\Char(F)\neq 2$. 
We have 
$T(\M_1,\ldots ,\M_n) = 0$  
if two of $\M_i$'s are unipotent algebraic groups over $F$ (\cite{IR}, Thm.\ 5.5.1). 
\end{itemize}
By the very construction, 
for any finite point $x$ over $F$, 
the product 
$T(\M_1,\ldots,\M_n)(x)$ evaluated at $x$ 
is a quotient of the Mackey product $(\M_1\otimesM \cdots \otimesM \M_n)(x)$. 
More precisely, the product $T(\M_1,\ldots,\M_n)$ 
is defined to be 
the Nisnevich sheafification $\L_{\mathrm{Nis}}^{\infty}$ 
of a quotient $\L^{\infty}$ of the product 
$\L := \M_1\otimes \cdots \otimes \M_n$ 
whose underlying Mackey functor 
is the Mackey product $\M_1\otimesM \cdots \otimesM \M_n$. 
However, 
an isomorphism $\L^{\infty}(x) \isomto \L_{\mathrm{Nis}}^{\infty}(x)$ exists 
since any Nisnevich covering of $\Spec F$ refines a trivial covering. 

Next we show the following theorem a part of the main theorems:

\begin{thm}
\label{thm:S}
Let $F$ be a finite field with $\Char(F)\neq 2$ and 
$\M_1,\ldots ,\M_n$ reciprocity functors over $F$. 
If two of $\M_i$'s are (non-trivial) algebraic groups, 
then the underlying Mackey functor of $T(\M_1,\ldots ,\M_n)$ is trivial. 
In other words, $T(\M_1,\ldots ,\M_n)(x) = 0$ for any finite point $x\to \Spec F$. 
\end{thm}
\begin{proof}
Here we show 
$T(G_1,G_2, \M) =0$  
for any algebraic groups $G_1,G_2$ and a reciprocity functor $\M$.  
The proof of the assertion 
$T(G_1,G_2, \M_1,\ldots ,\M_n) = 0$ 
for $n>1$ 
is exactly same as in the case of $n=1$. 
(Note also that since 
the product of the reciprocity functors does not satisfy 
the associativity (R1),  
we cannot reduce the assertion to $n=2$). 
The algebraic group $G_i$ has a decomposition 
\begin{equation}
\label{eq:dec}
0 \to U_i \to G_i \to A_i \to 0,
\end{equation}
where $U_i$ is a unipotent algebraic group 
and $A_i$ is a semi-abelian variety. 
The short exact sequence (\ref{eq:dec}) 
gives an admissible exact sequence 
in the category of reciprocity functors (\cite{IR}, Lem.\ 3.2.12). 
From the right exactness (R2) (and (R1)), 
we obtain 
the following admissible exact sequences 
\begin{equation}
\label{eq:diagram}
\vcenter{
  \xymatrix@R=5mm{
  T(U_1, U_2, \M)   \ar[d] & & T(A_1, U_2,\M)\ar[d]\\
  T(U_1, G_2,\M) \ar[d]\ar[r] 
    &T(G_1, G_2,\M) \ar[r] & T(A_1, G_2,\M) \ar[d]\ar[r] & 0 \\
  T(U_1, A_2,\M) \ar[d] & & T(A_1, A_2,\M)\ar[d] \\
  0  & & 0& .
  }
}
\end{equation}
By (R3), we obtain $T(U_1, U_2, \M) = 0$ as reciprocity functors. 
Since there is the surjection 
$(A_1\otimesM A_2\otimesM \M) \surj T(A_1, A_2, \M)$ 
of Mackey functors, 
and the canonical isomorphism 
$A_1\otimesM A_2\otimesM \M \simeq (A_1\otimesM A_2) \otimesM \M$, 
we obtain $T(A_1, A_2, \M) = 0$ as a Mackey functor by Kahn's theorem (\ref{eq:Kahn}). 
Similarly, Lemma \ref{lem:UA} implies 
$(U_1\otimesM A_2 \otimesM \M)  = T(U_1,A_2,\M) = 0$ 
and 
$(A_1\otimesM U_2 \otimesM \M)  = T(A_1,U_2,\M) = 0$.  
The assertion $T(G_1,G_2,\M) = 0$ follows from the above diagram. 
\end{proof}
\begin{crl}
Let $\M_1,\ldots , \M_m$ be reciprocity functors over a finite field $F$. 

\sn
$\mathrm{(i)}$ $T(\K_n^M,\M_1,\ldots, \M_m) =0$ as a Mackey functor if $n\ge 2$. 

\sn
$\mathrm{(ii)}$ $T(\Omega^n, \M_1,\ldots ,\M_m) = 0$ as a Mackey functor if $n\ge 1$. 

\sn
$\mathrm{(iii)}$ 
Let $X$ be a quasi-projective smooth and geometrically connected variety over $F$. 
Assume that there is a smooth projective and connected variety $\ol{X}$ over 
the algebraically closed field $\ol{F}$ containing $X_{\ol{F}} := X\otimes_F \ol{F}$ 
as an open subscheme. 
Then we have  
$T(h_0(X)^0, \M_1,\ldots, \M_m) =0$ as a Mackey functor 
if one of $\M_i$'s is an algebraic group, 
where $h_0(X)^0$ is the kernel of the degree map 
$\deg:h_0(X) \to h_0(\Spec F) \simeq \Z$ in the category of reciprocity functors. 
\end{crl}
\begin{proof}
By the computations of the products (\cite{IR}, Sect.\ 5) 
we have 
$T(\Gm^{\times n}) \simeq \K_n^M$ as Mackey functors 
and $T(\Ga, \Gm^{\times n}) \simeq \Omega^n$ as reciprocity functors.  
Here we used the notation $T(\M, \N^{\times n}) := T(\M, \overbrace{\N,\ldots,\N}^{n\mbox{-times}})$ 
for some reciprocity functors $\M$ and $\N$. 
For any finite point $x$, there are surjections 
\begin{align*}
  ((\Gm)^{\otimesM\,n}\otimesM \M_1 \otimesM \cdots \otimesM \M_m)(x) 
   &\surj (T(\Gm^{\times n}) \otimesM \M_1 \otimesM \cdots \otimesM \M_m) (x) \\
&\surj T(\K_n^M, \M_1,\ldots, \M_n)(x).
\end{align*}
The assertion (i) follows from (\ref{eq:Kahn}). 
The proof of (ii) is same as (i) 
(by using Lemma \ref{lem:UA} instead of (\ref{eq:Kahn})). 
For (iii), we show only $T(h_0(X)^0, G) (x) = 0$ 
for $x = \Spec F$ and an algebraic group $G$ over $F$. 
The general case can be proved by the same way. 
The proof is essentially same as in the one of 
Proposition 3.7 in \cite{Akhtar04c}. 
To show the assertion, 
it is enough to show 
that the image 
$\ol{\{a, b\}}_{x/x}$ of 
$\{a,b\}_{x/x}$ in $T(h_0(X)^0, G) (x)$ 
is trivial. 
The element $a\in h_0(X)(x) = h_0(X)$ is represented by 
a $0$-cycle $\sum_{i=1}^n n_i x_i$ ($n_i\neq 0$) 
on $X$. 
By the Bertini-type theorem of Poonen, 
there exists a smooth projective curve $\Cbar \subset \ol{X}$ 
such that 
$\Cbar$ contains $x_i$ for all $i$. 
Denoting $C := \Cbar \cap X$, $\ol{\{a, b\}}_{x/x}$ 
is in the image of the 
canonical map 
$T(h_0(C)^0, G)(x) \to T(h_0(X)^0, G)(x)$ induced by the inclusion $C\inj X$. 
Since $h_0(C)^0$ is isomorphic to the generalized Jacobian variety 
$J_{C}$ of Rosenlicht which is a semi-abelian variety, 
we obtain $\ol{\{a,b\}}_{x/x} =0$ from Theorem \ref{thm:S}. 
(The field $F$ may have $\Char(F) = 2$, 
since we do not need (R3) to show the assertion.) 
\end{proof}

Finally, we consider the Mackey product of algebraic groups over a finite field. 
\begin{thm}
\label{thm:M}
Let $G_1,\ldots ,G_n$ be algebraic groups 
over a finite field $F$. 
Then the group $(G_1\otimesM \cdots \otimesM G_n) (x)$ is finite 
for any finite point $x$ over $F$. 
\end{thm}
\begin{proof}
The algebraic group $G_i$ has a decomposition 
$$
0 \to U_i \to G_i \to A_i \to 0
$$ 
with unipotent part $U_i$ and a semi-abelian variety $A_i$ over $F$. 
Replacing the product of reciprocity functors 
by the Mackey products in (\ref{eq:diagram}),  
Kahn's theorem (\ref{eq:Kahn}) and Lemma \ref{lem:UA} 
gives a surjection 
$(U_1\otimesM\cdots \otimesM U_n)(x) \surj (G_1\otimesM \cdots \otimesM G_n)(x)$. 
Thus it is enough to show the assertion for 
$G_i = U_i$ for all $i$. 
A composition series of the unipotent group $G := G_i$: 
$$
  0 = G^r \subset \cdots \subset G^1 \subset G,
$$
each $G^i/G^{i+1} \simeq \Ga$. 
By the right exactness (M2), 
we may assume $G_i = \Ga$ for all $i$ and $x = \Spec F$ 
without loss of generality. 
We show the finiteness of $(\Ga)^{\otimesM\, n}(x)$ by 
induction on $n$. 
The case of $n=1$, there is noting to show. 
For $n>1$, we assume that $(\Ga)^{\otimesM\, (n-1)}(x)$ 
is finite. 
The group $(\Ga)^{\otimesM\, n}(x)$ has a structure of 
an $F$-vector space given by 
$a\{a_1,\ldots , a_n\}_{x'/x} := \{j^{\ast}(a)a_1,\ldots , a_n\}_{x'/x}$, 
for any $a\in F$ and a symbol $\{a_1,\ldots ,a_n\}_{x'/x}$ on a finite point $j:x'\to x = \Spec F$. 
Consider a subspace $I(x)$ of 
$(\Ga)^{\otimesM\, n}(x)$ generated by 
the elements of the form 
$$
\{1, a_2,\ldots, a_n\}_{x'/x} - 
\{a_2\cdots a_n, 1, \ldots, 1\}_{x'/x}.
$$ 
By identifying the canonical isomorphism 
$(\Ga)^{\otimesM\,n} \simeq \Ga\otimesM (\Ga)^{\otimesM\, (n-1)}$ by (M1) 
and $j^{\ast}(1) = 1 \in \Ga(x)$, 
we have 
\begin{align*}
\{1, a_2,\ldots, a_n\}_{x'/x}  
  &= \{j^{\ast}(1), \{a_2,\ldots, a_n\}_{x'/x'}\}_{x'/x} \\
  &= \{1, j_{\ast}\{a_2,\ldots, a_n\}_{x'/x'}\}_{x/x} \quad \mbox{by the projection formula}\\
  &= \{1, \{a_2,\ldots, a_n\}_{x'/x}\}_{x/x}\quad \mbox{by (M3)}.  
\end{align*}
By the induction hypothesis, the elements of this form are finite. 
On the other hand, the projection formula implies  
$$
\{a_2\cdots a_n, 1,\ldots , 1\}_{x'/x} 
= \{j_{\ast}(a_2 \cdots a_n), 1,\ldots , 1\}_{x/x}. 
$$ 
Thus the subspace $I(x)$ is finite. 
Next we consider a subspace $S(x)$ of 
the quotient $Q(x) := (\Ga)^{\otimesM\, n}(x)/I(x)$ 
generated by symbols of the form $\ol{\{a_1,\ldots, a_n\}}_{x/x}$. 
Here we denote by $\ol{\{a_1,\ldots, a_n\}}_{x'/x}$ 
the image of $\{a_1,\ldots , a_n\}_{x'/x}$ in the quotient $Q(x)$. 
It is easy to see that the subspace $S(x)$ is finite. 
For any symbol $\ol{\{a_1,\ldots , a_n\}}_{x'/x}$ in $Q(x)$ 
on a finite point $j:x' \to x$, 
we have 
\begin{align*}
\ol{\{a_1,\ldots , a_n\}}_{x'/x} &= j_{\ast}(\ol{\{a_1,\ldots , a_n\}}_{x'/x'}) \quad \mbox{by (M3)}\\
                   &= j_{\ast} (a_1\ol{\{1, a_2, \ldots ,a_n\}}_{x'/x'}) 
\quad \mbox{because of $\ol{\{a_1,\ldots , a_n\}}_{x'/x'} \in Q(x')$}\\
                   &= \ol{\{a_1\cdots a_n,1,\ldots , 1\}}_{x'/x} \quad \mbox{by (M3)}\\
                   &= \ol{\{j_{\ast}(a_1\cdots a_n),1,\ldots , 1\}}_{x/x} \quad \mbox{by the projection formula}.
\end{align*}
Thus we obtain $Q(x) = S(x)$ and the assertion follows from it. 
\end{proof}

\section{Applications}
\label{sec:app}
Let $X$ be a smooth (and connected) variety over a finite field $F$. 
Assume that there is a smooth compactification $\ol{X}$ of $X$, that is, 
a projective smooth variety which contains $X$ as an open subscheme.  
Let $D$ be an effective divisor on $\ol{X}$ with support in $\ol{X} \ssm X$.  
To consider the ramification along the divisor on $D$ 
here we recall the relative Chow group 
$\CH_0(X,D)$  of the pair $(X,D)$ 
(\Cf \cite{EK12}, Set.\ 8.1, see also \cite{Rus10}, Sect.\ 3.4 and 3.5). 
Define 
$$
 \CH_0(X,D) := \Coker\left(\div:\bigoplus_{\phi:C \to X}P_C(\ol{\phi}^{\ast}D) \to Z_0(X)\right), 
$$
where 
the direct sum runs over the normalization $\phi:C\to X$ of a curve in $X$, 
$Z_0(X)$ is the group of $0$-cycles on $X$, 
$\ol{\phi}:\ol{C} \to \ol{X}$ is the extension of the map $\phi$ to 
the smooth compactification $\ol{C}$ of $C$, 
the map 
$\div$ is given by the divisor map on each curve $C$ and 
$$
  P_{C}(\ol{\phi}^{\,\ast}D) := \{f \in F(C)^{\times}\ |\ f\equiv 1 \mod \ol{\phi}^{\,\ast}D + (\ol{C}\ssm C)_{\red}\}. 
$$ 
Putting $X_x := X\times_{\Spec F} x$ and denoting by $D_x$ the 
pull-back of $D$ to $\ol{X}_x := \ol{X}\times_{\Spec F}x$ 
for any finite point $x\to \Spec F$, the assignment
$$
  \sCH_0(X,D):x\mapsto \CH_0(X_x, D_x)
$$
gives a Mackey functor $\sCH_0(X,D)$. 
As an application of the main theorem, 
we show the following finiteness theorem 
of the relative Chow group:

\begin{thm}
\label{thm:CH}
Let  $X_1,\ldots ,X_n$ 
be smooth and connected curves
over a finite field $F$ with $X_i(F)\neq \emptyset$ 
and put $X := X_1\times \cdots \times X_n$.  
For  an effective divisor $D$ on $\ol{X} := \ol{X}_1\times \cdots \times \ol{X}_n$ 
with support in $\ol{X}\ssm X$,  
the kernel of the degree map $\CH_0(X,D)^0 := \Ker(\deg:\CH_0(X,D) \to \Z)$ 
is finite. 
\end{thm}
\begin{proof}
We show the case $n=2$ (the proof is same for $n>2$). 
For sufficiently large divisors $D_i$ on $\ol{X}_i$ for $i = 1,2$, 
there is a surjection $\CH_0(X,D_1\times X_2  + X_1 \times D_2) \surj \CH_0(X,D)$. 
Thus we may assume that $D$ is a divisor of the form 
$D_1\times X_2  + X_1 \times D_2$. 
Now we consider the map  
\begin{equation}
\label{eq:prod}
  \psi: (\sCH_0(X_1,D_1) \otimesM \sCH_0(X_2,D_2)) (\Spec F) \to \CH_0(X,D)
\end{equation}
defined by  
$\{a_1,a_2\}_{x/\Spec F} \mapsto (j_x)_{\ast}(p_1^{\ast}a_1 \cap p_2^{\ast}a_2)$, 
where $p_i:\ol{X} = \ol{X}_1\times \ol{X}_2 \to \ol{X}_i$ 
is the projection, 
$j_x:X_x \to X$ is given by the base change 
to $x$ 
and $\cap$ is the internal product defined similarly to the ordinal Chow group of $0$-cycles. 
We show that the map $\psi$ is surjective. 
Take a cycle $[x]$ as a generator of $\CH_0(X,D)$ 
which is represented by a closed point $x$ on $X$ 
and is a finite point $j_x:x\to \Spec F$. 
By the definition of $\psi$, 
the push-forward map on the relative Chow group 
and the norm map on the Mackey product are compatible 
as in the following commutative diagram:
$$
 \xymatrix{
  (\sCH_0(X_1, D_1)\otimesM \sCH_0(X_2,D_2)) (x) 
    \ar[d]_{(j_{x})_{\ast}}\ar[r]^-{\psi_{x}} & \CH_0(X_{x},D_{x}) 
    \ar[d]^{(j_{x})_{\ast}}\\ 
   (\sCH_0(X_1, D_1)\otimesM \sCH_0(X_2,D_2)) (\Spec F) \ar[r]^-{\psi} & \CH_0(X,D).
}
$$
Thus to show the surjectivity of $\psi$ we may assume that $x$ 
is an  $F$-rational point. 
The point $x$ is determined by maps $x_i \to X_i$. These points give 
$\psi(\{[x_1],[x_2]\}_{x/x}) = [x]$ and thus $\psi$ is surjective. 

From the assumption $X_i(F)\neq \emptyset$, 
there exists a decomposition 
$\sCH_0(X_i,D_i) \simeq \Z\oplus J_{X_i,D_i}$ 
by the generalized Jacobian variety $J_{X_i,D_i}$ 
of the pair $(X_i,D_i)$  (\cite{AGCF}). 
According to this decomposition we obtain
$$
\sCH_0(X_1,D_1) \otimesM \sCH_0(X_2,D_2) \simeq  
 \Z \oplus J_{X_1,D_1} \oplus J_{X_2,D_2} 
\oplus (J_{X_1,D_1}\otimesM J_{X_2,D_2})  
$$
by (M1). 
From the product (\ref{eq:prod}), 
there exists a surjection 
$$
  J_{X_1,D_1}(F) \oplus J_{X_2,D_2}(F) 
\oplus (J_{X_1,D_1}\otimesM J_{X_2,D_2})(\Spec F) \surj \CH_0(X,D)^0. 
$$
The left is finite by Theorem \ref{thm:M} 
and so is $\CH_0(X,D)^0$. 
\end{proof}

Let $X$ be the product of curves over $F$, 
and $D$ as in the above theorem (Thm. \ref{thm:CH}).  
For each normalization $\phi:C\to X$ of a curve in $X$, 
we have a divisor 
$D_C := \ol{\phi}^{\ast}(D) + (\ol{C} \ssm C)_{\red}$ 
on $\ol{C}$. 
The category of \'etale coverings of $X$ 
with ramification bounded by the collection of divisors $(D_C)_{\phi:C\to X}$ 
forms a Galois category and gives a fundamental group $\pi_1(X,D)$ (\cite{HH09}, Lem.\ 3.3). 
For each such $C$ and a point $x \in \Cbar \ssm C$ 
there is the canonical map 
$G_{C,x}^{\ab} := \Gal(F(C)_x^{\ab}/F(C)_x) \to \pi_1(X)^{\ab}$, 
where $F(C)_x^{\ab}$ is the maximal abelian extension of 
the completion $F(C)_x$ at $x$. 
By the very definition of the coverings, we have  
$$
  \Coker\left(\bigoplus_{\phi:C \to X}\bigoplus_{x \in \Cbar\ssm C} G_{C,x}^{\ab,m_{x}(D_C)} 
\to \pi_1(X)^{\ab}\right) \surj \pi_1(X,D)^{\ab},
$$
where $G_{C,x}^{\ab, m}$ is the $m$-th (upper numbering) ramification subgroup of 
$G_{C,x}^{\ab}$ (\cite{Ser68}, Chap.\ IV, Sect.\ 3) 
and $m_{x}(D_C)$ is the multiplication of the divisor $D_C$ at $x$. 
Using the idele theoretic description of the relative Chow group 
$$
  \CH_0(X,D) \simeq 
\Coker\left(\bigoplus_{\phi:C \to X} F(C)^{\times} \to  
 Z_0(X) \oplus \bigoplus_{\phi:C \to X}\bigoplus_{x \in \Cbar\ssm C} F(C)_x^{\times} /U_{C,x}^{m_{C, x}} \right),
$$
where $U_{C,x}^m = 1 + \mathfrak{m}_{C,x}^m$ is the higher unit group of $F(C)_x$, 
local class field theory (see e.g., \cite{Ser68}, Chap.\ XV) induces the following commutative diagram:
$$
\xymatrix{
  Z_0(X)^0 \ar[d]_{\rho}\ar@{->>}[r] & \CH_0(X,D)^0 \ar[d]^{\rho_D} \\
  \pi_1(X)^{\ab,0} \ar@{->>}[r] & \pi_1(X,D)^{\ab,0}.
}
$$
Here 
$Z_0(X)^0$ is the kernel of the degree map $\deg:Z_0(X) \to \Z$, 
the left vertical map $\rho$ is the reciprocity map on $X$ 
and 
$\pi_1(X,D)^{\ab,0}$ is the geometric part of 
the abelian fundamental group 
(= the kernel of the canonical map $\pi_1(X,D)^{\ab} \to \pi_1(\Spec(F))^{\ab})$.
The image of the left reciprocity map $\rho:Z_0(X)\to \pi_1(X)^{\ab}$ is known to be dense 
(due to Lang \cite{Lang56}) 
and the image of $\rho_D$ is finite by Theorem \ref{thm:CH}.  
Therefore, the map $\rho_D$ is surjective 
and we obtain the following finiteness result. 

\begin{crl} 
Let $X$ and $D$ be as in Theorem \ref{thm:CH}. 
Then, the geometric part of the abelian fundamental group 
$\pi_1(X,D)^{\ab, 0}$ is finite.
\end{crl}

We conclude 
this note by referring to 
recent results in \cite{EK12}. 
H.~Esnault and M.~Kerz showed 
the finiteness of the group $\CH_0(X,D)^0$ 
for a smooth variety $X$ over a finite field 
and an effective Cartier divisor $D$ with support in the boundary 
$\ol{X} \ssm X$ 
only assuming the existence of some normal compactification $\ol{X}$ of $X$ 
 (\cite{EK12}, Thm.\ 8.1).  
This is an application of Deligne's 
finiteness theorem for $l$-adic Galois representations of function fields. 
Using this finiteness theorem 
instead of Theorem \ref{thm:CH} in the above arguments, 
we also obtain the finiteness of $\pi_1(X,D)^{\ab, 0}$ 
for the variety $X$.

%

\def\cprime{$'$}
\providecommand{\bysame}{\leavevmode\hbox to3em{\hrulefill}\thinspace}
\providecommand{\href}[2]{#2}

\noindent
Toshiro Hiranouchi \\
Department of Mathematics, Graduate School of Science, Hiroshima University\\
1-3-1 Kagamiyama, Higashi-Hiroshima, 739-8526 Japan\\
Email address: {\tt hira@hiroshima-u.ac.jp}

\end{document}